%% file: main.tex
\documentclass[pdflatex,sn-mathphys-num]{sn-jnl}

\usepackage{graphicx}
\usepackage{multirow}
\usepackage{amsmath,amssymb,amsfonts}
\usepackage{amsthm}
\usepackage{mathrsfs}
\usepackage[title]{appendix}
\usepackage{xcolor}
\usepackage{textcomp}
\usepackage{manyfoot}
\usepackage{booktabs}
\usepackage{algorithm}
\usepackage{algorithmicx}
\usepackage{algpseudocode}
\usepackage{listings}

\usepackage{tabularx}

\theoremstyle{thmstyleone}
\newtheorem{theorem}{Theorem}
\newtheorem{proposition}[theorem]{Proposition}

\theoremstyle{thmstyletwo}

\theoremstyle{thmstylethree}

\begin{document}

\title{Efficient Implementation of high-order isospectral symplectic Runge-Kutta schemes}

\author*[1]{\fnm{Clauson} \sur{Carvalho da Silva}}\email{clauson@puc-rio.br}

\author[1]{\fnm{Christian} \sur{Lessig}}\email{christian.lessig@ovgu.de}
\equalcont{These authors contributed equally to this work.}

\author[2]{\fnm{Carlos} \sur{Tomei}}\email{tomei@puc-rio.br}
\equalcont{These authors contributed equally to this work.}

\affil*[1]{\orgdiv{Institut f{\"u}r Simulation und Graphik}, \orgname{Otto-von-Guericke-Universit{\"a}t}, \city{Magdeburg}, \country{Germany}}

\affil[2]{\orgdiv{Departmento de Matemática}, \orgname{PUC-Rio}, \city{Rio de Janeiro}, \country{Brazil}}

\abstract{Isospectral Runge-Kutta methods are well-suited for the numerical solution of isospectral systems such as the rigid body and the Toda lattice. More recently, these integrators have been applied to geophysical fluid models, where their isospectral property has provided insights into the long-time behavior of such systems. However, higher-order Isospectral Runge-Kutta methods require solving a large number of implicit equations. This makes the implicit midpoint rule the most commonly used due to its relative simplicity and computational efficiency. In this work, we introduce a novel algorithm that simplifies the implementation of general isospectral Runge-Kutta integrators. Our approach leverages block matrix structures to reduce the number of implicit equations per time step to a single one. This equation can be solved efficiently using fixed-point iteration. We present numerical experiments comparing performance and accuracy of higher-order integrators implemented with our algorithm against the implicit midpoint rule. Results show that, for low-dimensional systems, the higher-order integrators yield improved conservation properties with comparable computational cost. For high-dimensional systems, while our algorithm continues to show better conservation properties, its performance is less competitive, though it can be improved through parallelization.}


\maketitle

\input{intro}
\input{preliminaries}

\input{algorithm}

\input{numerics}

\input{conclusion}

\bibliography{references.bib}

\end{document}

%% file: intro.tex
\section{Introduction}

In numerical simulations of physical systems, simplicity, efficiency, and accuracy are primary concerns. It is also important that the resulting discrete numerical solutions reproduce the conservation laws of their continuous counterparts. Conservation laws allow for long-term time integrations and can be used as metrics to evaluate the accuracy of time-discrete solutions.   

Isospectral flows and Lie-Poisson systems have been extensively investigated since they arise in a number of idealized physical systems. Under a set of conditions, they can be related (see Section~\ref{sec:isospectral_Lie_Poisson}). We call the intersection isospectral Lie-Poisson systems. Their main property is isospectrality (constant spectrum), which is related to the preservation of Casimirs for the systems. Classical examples of isospectral Lie-Poisson systems are the ideal rigid body~\cite{Marsden1999} and the periodic Toda lattice in its Lax pair formulation~\cite{Toda1967, tomei2015toda}. Several algorithms have been designed to preserve the physical properties of the continuous system. 
In~\cite{deift1983ordinary}, QR-type algorithms were developed to discretize the (non-periodic) Toda flow.  
The authors in~\cite{calvo1997numerical}, showed that Runge-Kutta methods applied directly to the Lie-Poisson equation of isospectral systems do not preserve isospectrality, and they proceeded to develop Gauss-Legendre Runge-Kutta schemes that overcome this issue.

Recently,~\cite{modin2019lie} introduced Isospectral Symplectic Runge-Kutta (ISOSyRK) methods for the numerical integration of isospectral Lie-Poisson systems. The approach is based on using Lie-Poisson reconstruction to lift the systems from the Lie algebra, where the dynamics originally takes place, to canonical Hamiltonian systems on the associated phase space. There, the use of symplectic implicit Runge-Kutta methods for time-discretization results in an algorithm that preserves isospectrality and that is a Lie-Poisson integrator. Furthermore, the ISOSyRK methods have the same order as the underlying Runge-Kutta method which, together with symplecticity, bounds the Hamiltonian oscillations, making the time integrator long-term stable.

Despite their favorable properties, ISOSyRKs require for each time step a large number of implicit equations to be solved when the order of the integrator is high. In~\cite{viviani2019minimal,da2022variational} the authors show that the algorithms can be greatly simplified for the case of symplectic diagonally implicit Runge-Kutta methods (symplectic DIRKs), which are scaled compositions of the implicit midpoint rule \cite{hairer2006geometric}. These algorithms are referred to as ISOSyDIRKs.

In the present work, we introduce a novel algorithm for the implementation of general ISOSyRKs. We refer to it as the block algorithm for ISOSyRKs. It enjoys simplicity and computational complexity comparable to the ISOSyDIRKs. The block algorithm facilitates the implementation of high-order ISOSyRKs (better accuracy) with a small number of stages (better efficiency), e.g., the $6^{th}$-order 3-stage Gauss method. Furthermore, the ISOSyDIRKs have sections that need to be executed in a serial manner, while the proposed block algorithm can be fully parallelized. In~\cite{cifani2023efficient}, an isospectral system is implemented as a discretization of the spherical barotropic fluid model. Due to the large size of the matrices ($n \sim 1000$), their implementation heavily relies on a parallel computing infrastructure. The authors use the ISOSyRK associated with the implicit midpoint rule as their time-integrator. This parallel computing power can be used to
implement the block algorithm for higher-order ISOSyRK schemes, resulting in better conservation properties and comparable efficiency. This will be the subject of a following work.

The remainder of the paper is organized as follows. In section~\ref{sec:preliminaries}, we provide a brief summary on isospectral Lie-Poisson systems and the time-integrators from~\cite{modin2019lie,viviani2019minimal,da2022variational}. Section~\ref{sec:block_algorithm} describes the block algorithm for ISOSyRKs. In Section~\ref{sec:numerical}, we show numerical results for the ideal rigid body and the periodic Toda lattice. We compare the results from the implementation that uses the block algorithm and the implementations from~\cite{viviani2019minimal,da2022variational}. Directions for future work are discussed in Section~\ref{sec:conclusion}.

%% file: preliminaries.tex
\section{Preliminaries}
\label{sec:preliminaries}

\subsection{Isospectral Lie-Poisson systems on matrix Lie algebras}
\label{sec:isospectral_Lie_Poisson}

General isospectral systems can be described by matrix differential equations of the form,
\begin{align}
  \label{eq:isospectral_conceputal}
  \dot{A} = \big[ B(A) , A \big] , \quad A(0) = A_0,
\end{align}
where $[\cdot,\cdot]$ is the matrix commutator, $A$ belongs to a subset $S$ of a matrix Lie algebra $\mathfrak{g}$ and $B:S \to \mathfrak{g}$ possesses properties that guarantee that the right-hand side of Equation~\eqref{eq:isospectral_conceputal} belongs to $S$. The solution $A(t)$ is known to be isospectral, i.e., its eigenvalues do not depend on $t$. See~\cite{calvo1997numerical,modin2019lie} for a detailed discussion of the properties of these systems.

Lie-Poisson systems are Hamiltonian systems on dual Lie algebras $\mathfrak{g}^*$ governed by
\begin{align}\label{eq:lie_poisson}
  \dot{\mu} 
  = \text{ad}^*_{\nabla \eta(\mu)} \mu, \quad \mu\in\mathfrak{g}^*
\end{align}
where $\eta : \mathfrak{g}^* \to \mathbb{R}$ is a Hamiltonian function and $\text{ad}^*$ is the coadjoint action of $\mathfrak{g}$ on $\mathfrak{g}^*$. For matrix Lie algebras, $\text{ad}^*$ is given by
\begin{align*}
    \left \langle\text{ad}^*_{\alpha} \mu, \beta \right \rangle
    =
    \left \langle \mu, [\alpha,\beta] \right \rangle
\end{align*}
where $\langle \cdot, \cdot \rangle$ is the pairing between $\mathfrak{g}^*$ and $\mathfrak{g}$ and $[\cdot, \cdot]$ is the usual matrix commutator. 

Isospectral systems and Lie-Poisson systems on matrix Lie algebras can be related under a set of conditions. Using the Frobenius inner product on $\mathfrak{g}$, defined by $\langle \alpha, \beta \rangle = \mathrm{tr}(\alpha^{\dagger} \beta)$, with $^\dagger$ being the conjugate transposition, the Lie algebra and its dual can be identified. Furthermore, if $\mathfrak{g}$ is a reductive Lie algebra, that is, $[\mathfrak{g},\mathfrak{g}^\dagger] \subset \mathfrak{g}$, then Equation~\eqref{eq:lie_poisson} can be written in the form of Equation~\eqref{eq:isospectral_conceputal} with $B(\cdot) = \nabla \eta(\cdot)^\dagger$. This is the case for the ideal rigid body treated in more detail in Section~\ref{sec:numerical_rb}.
Conversely, consider an isospectral system, where the map $B$ can be written as $B(\cdot) = \nabla \eta(\cdot)^\dagger$ for a Hamiltonian $\eta:\mathfrak{gl}(n) \to \mathbb{R}$ that is constant on the affine spaces given by the translations of the orthogonal complement of $S$. Such isospectral system can be extended to a Lie-Poisson system on $\mathfrak{gl}(n)$. This is the case of the periodic Toda flow, which we treat in Section~\ref{sec:numerical_toda}.

In the remainder of this work, we call a system isospectral Lie-Poisson when its dynamics are given by
\begin{align}
  \label{eq:isospectral_lie_poisson}
  \dot{W} = \big[ B(W) , W \big] , \quad W(0) = W_0
\end{align}
where $W$ belongs to a reductive matrix Lie algebra $\mathfrak{g}$, identified with its dual under the Frobenius inner product, and $B(\cdot) = \nabla \eta(\cdot)^\dagger$ for a given Hamiltonian $\eta:\mathfrak{g} \to \mathbb{R}$.

\subsection{ISOSyRK methods}
The Isospectral Symplectic Runge-Kutta (ISOSyRK) methods are a class of time stepping schemes introduced in~\cite{modin2019lie} to numerically solve Equation~\eqref{eq:isospectral_lie_poisson}. 
Instead of applying symplectic Runge-Kutta schemes directly to Equation~\eqref{eq:isospectral_lie_poisson}, the core idea of the ISOSyRK methods is to use symplectic Runge-Kutta schemes on the canonical Hamiltonian systems obtained from the isospectral Lie-Poisson systems through reconstruction to the associated phase space~\cite[Chap. 13]{Marsden1999}.

An isospectral Lie-Poisson system on a Lie algebra $\mathfrak{g}$ of a matrix Lie group $G$ with Hamiltonian $\eta:\mathfrak{g} \to \mathbb{R}$ can be pulled back to a canonical Hamiltonian system on the phase space $T^*G$ using a momentum map. This results in Hamilton's equations of the form
\begin{align}
  \label{eq:hamilton_eq_iso}
  \begin{aligned}
  \dot{g} &= g \nabla \eta(g^\dagger p) = g B(g^\dagger p)^\dagger  
  \\[4pt] 
  \dot{p} &= -p \nabla \eta(g^\dagger p)^\dagger = - p B(g^\dagger p).
  \end{aligned}
\end{align}
The system defined above is equivalent to the isospectral Lie-Poisson system since a solution $(g(t), p(t))$ defines a solution $W(t) = g(t)^{\dagger}p(t)$ for the isospectral Lie-Poisson system. For a fixed time-step $h$, applying an $s$-stage symplectic Runge Kutta scheme defined by the Butcher Tableau
\begin{align*}
\begin{array}{c|cccc}
 c_1 & a_{11}  & \ldots & a_{1s} \\[3pt]
 c_2 & a_{21} & \ldots & a_{2s}  \\[3pt]
 \vdots &  \vdots &  \vdots & \vdots \\[3pt]
 c_s & a_{s1} & \cdots  & a_{ss}\\[3pt]
\hline\rule{0pt}{1.01\normalbaselineskip}
    &  b_1  & \cdots & b_s \\[3pt]
\end{array}
\end{align*}
to the system in Equation~\eqref{eq:hamilton_eq_iso} results in $2s$ implicit equations that define intermediate points $G_i$ and $P_i$
\begin{subequations}
    \label{eq:rk:hamilton}
    \begin{align}
    G_i = g_0 + h\sum_{j=0}^s a_{ij} G_j B(M_j)^\dagger        
    \end{align}
    \begin{align}
    P_i = p_0 - h\sum_{j=0}^s a_{ij} P_j B(M_j)        
    \end{align}
\end{subequations}
where $M_{i} = G_i^\dagger P_i$, with $i = 1,\ldots,s$, and update equations
\begin{subequations}
    \begin{align}
    g_1 = g_0 + h\sum_{i=0}^s b_{i} G_i B(M_j)^T        
    \end{align}
    \begin{align}
    p_1 = p_0 - h\sum_{i=0}^s b_{i} P_i B(M_j).        
    \end{align}
\end{subequations}
The above numerical scheme is shown in~\cite{modin2019lie} to descend to a scheme on $\mathfrak{g}^* \simeq \mathfrak{g}$ given by
\begin{align}
    \label{eq:ISOSyRK_scheme}
    \begin{aligned}
    X_i &= -h\left(W_0 + \sum_{j=1}^s a_{ij} X_j\right) B(M_i)\\ 
    Y_i &= h B(M_i)\left( W_0 + \sum_{j=1}^s a_{ij} Y_j \right) \\
    K_{ij} &= h B(M_i)\left(\sum_{j'=1}^s \left(a_{ij'} X_j + a{jj'} K_{ij'}\right) \right)\\
    M_i &= W_0 + \sum_{j=1}^s a_{ij} \left(X_j + Y_j + K_{ij}\right)\\
    W_1 &= W_0 + h \sum_{i = 1}^s b_i [B(M_i),M_i]
    \end{aligned} 
\end{align}
for $i,j = 1, \ldots,s$. They are called ISOSyRK methods for isospectral Lie-Poisson systems and possess a number of desirable properties:
\begin{itemize}
    \item Equation~\eqref{eq:ISOSyRK_scheme} has the same order as the underlying Runge-Kutta method. The order of the numerical method is related to the conservation of properties such as energy. 
    
    \item They are isospectral, i.e, the eigenvalues of the numerical solution $W_k$ are independent of $k$. This is a fundamental property since isospectrality is related to the conservation of Casimirs of the continuous systems. 
    
    \item It is a Lie-Poisson integrator, i.e, the numerical solution stays on the coadjoint orbits of $\mathfrak{g}^*$.
    
\end{itemize}

\subsection{ISOSyDIRK methods}
\label{sec:isosydirks}
Implementing an $s$-stage ISOSyRK with the algorithm described by Equation~\eqref{eq:ISOSyRK_scheme} requires solving $3s + s^2$ implicit equations. A simpler implementation is obtained when the underlying Runge-Kutta method is the implicit midpoint rule. In~\cite{viviani2019minimal}, the authors showed that manipulating Equation~\eqref{eq:ISOSyRK_scheme} for the implicit midpoint rule results in an algorithm that requires solving a unique implicit equation of the form 
\begin{align}
\label{eq:ISOSyDIRK_implicit}
	W_0 = \left(\mathrm{Id} - \frac{1}{2} B(M_1) \right) M_1 \left( \mathrm{Id} + \frac{1}{2} B(M_1) \right),
\end{align}
where $\mathrm{Id}$ is the identity matrix, followed by an explicit update equation
\begin{align}
W_1 = \left(\mathrm{Id} + \frac{1}{2} B(M_1) \right) M_1 \left( \mathrm{Id} - \frac{1}{2} B(M_1) \right).
\end{align}

The implicit midpoint rule is a $2^{nd}$-order scheme. The authors in~\cite{da2022variational} introduced ISOSyDIRKs, which generalize the simplicity of this algorithm for the case where the underlying Runge-Kutta scheme is a diagonally implicit symplectic Runge-Kutta method. This allows for simpler implementations of higher-order schemes when compared to Equation~\eqref{eq:ISOSyRK_scheme}. ISOSyDIRK methods require $s$ pairs of implicit-explicit equations of the form
\begin{align*}
    W_0^i = \left(\mathrm{Id} - \frac{1}{2} B(M_0^{i+1}) \right) M_0^{i+1} \left( \mathrm{Id} + \frac{1}{2} B(M_0^{i+1}) \right)\\
    W_0^{i+1} = \left(\mathrm{Id} + \frac{1}{2} B(M_0^{i+1}) \right) M_0^{i+1} \left( \mathrm{Id} - \frac{1}{2} B(M_0^{i+1}) \right)\\
\end{align*}
to be solved for $i = 0, \ldots, s - 1$, $W_0^0 = W_0$ and $W_0^s = W_1$.  The authors in~\cite{da2022variational} also show how such algorithms can be derived from a discrete variational principle.

An interesting aspect of the algorithms introduced in~\cite{viviani2019minimal,da2022variational} is the appearance of the Cayley transform
\begin{align}
    \label{eq:cayley}
    \mathrm{cay}(\nu) = \left(\mathrm{Id} - \frac{1}{2}\nu \right)^{-1}\left( \mathrm{Id} + \frac{1}{2}\nu\right).
\end{align}
for $\nu \in \mathfrak{g}$ . Analogous to the exponential map, the Cayley transform maps $\mathfrak{g}$ to the matrix Lie Group $G$ and it allows for the description of the ISOSyDIRKs on the cotangent bundle level.
For instance, for the implicit midpoint rule, if we set initial conditions $(g_0, p_0) \in T^*G$ such that $W_0 = g_0^\dagger p_0$, we can define $g_1 = g_0 \mathrm{cay}(hB(M_1)^\dagger)$ and $p_1 = p_0 \mathrm{cay}(hB(M_1))^{-1}$. It follows that $W_1 = g_1^{\dagger}p_1$ and analogous results are true for the more general ISOSyDIRKs. This unexpected appearence of the Cayley transform in the ISOSyDIRKs opens a series of questions such as:
\begin{enumerate}
    \item Are there other maps, analogous to the Cayley transform that operate in the background of more general ISOSyRKs?
    \item If yes, can they be used to simplify the implementation of more general ISOSyRKs?
    \item Can they also be used to derive discrete variational formulations for more general ISOSyRKs, similar to those in~\cite{da2022variational, gawlik2011geometric}? 
\end{enumerate}
The results in~\cite{stern2024quadratic} seem to suggest a negative answer for those questions. However, in the next section, we define an algorithm for general ISOSyRKs where the main ideas were inspired by the above questions. This novel algorithm has similar simplicity to the ISOSyDIRKs, despite its different formulation, and facilitates the implementation of higher-order schemes such as the $2$ and $3$-stage Gauss collocation methods.

%% file: algorithm.tex
\section{Block Algorithm for ISOSyRKs}
\label{sec:block_algorithm}

For an $s$-stage ISOSyRK, the general algorithm proposed by~\cite{modin2019lie} in Equation~\eqref{eq:ISOSyRK_scheme} requires that $3s+s^2$ variables are found implicitly before the update step. For the case of ISOSyDIRKs, the algorithms proposed in~\cite{da2022variational,viviani2019minimal} require that $s$ implicit equations are solved before the update step, where $s$ is the number of stages of the associated symplectic diagonally implicit Runge-Kutta scheme. Using block matrices, we introduce an algorithm, that allows us to find the intermediate points with one unique implicit equation. We refer to it as the block algorithm for ISOSyRKs. It can be seen as a generalization of the algorithm proposed in ~\cite{viviani2019minimal} and it simplifies the implementation of general ISOSyRK methods, specially those of higher order. Next, we present the block algorithm for ISOSyRKs.

Consider a $n$-dimensional isospectral Lie-Poisson system. We can write Equation~\eqref{eq:rk:hamilton} as

\begin{align}
\label{eq:rk:hamilton:matrix}
    \begin{aligned}
    \mathbf{G} &= \mathbf{g_0} + h\mathbf{G}\mathbf{B}^\dagger\mathbf{A}^T \\        
    \mathbf{P} &= \mathbf{p_0} - h\mathbf{P}\mathbf{B}\mathbf{A}^T.
    \end{aligned}
\end{align}
       
We have:
\begin{itemize}
    \item $\mathbf{G}$ and $\mathbf{P}$ are dense block row matrices, whose columns are the intermediate points $G_i$ and $P_i$, i.e., 
    \begin{align*}
    \mathbf{G} = 
    \begin{pmatrix}    G_1 & \cdots & G_s   \end{pmatrix} \quad \text{and} \quad \mathbf{P} = 
    \begin{pmatrix}    P_1 & \cdots & P_s   \end{pmatrix}.  
    \end{align*}
    Since $G_i, P_i \in \mathbb{R}^{n \times n}$, we have that $\mathbf{G}, \mathbf{P} \in \mathbb{R}^{n \times sn}$.
    
    \item $\mathbf{g_0}$ and $\mathbf{p_0}$ are dense block row matrices where repetitions of $g_0$ and $p_0$ define their columns, respectively, i.e.,  
    \begin{align*}
    \mathbf{g_0} = 
    \begin{pmatrix}    g_0 & \ldots & g_0   \end{pmatrix} \quad \text{and} \quad \mathbf{p_0} = 
    \begin{pmatrix}    p_0 & \ldots & p_0   \end{pmatrix}.  
    \end{align*}
    We have $g_0, p_0 \in \mathbb{R}^{n \times n}$, and therefore $\mathbf{g_0}, \mathbf{p_0} \in \mathbb{R}^{n \times sn}$.
    
    \item $\mathbf{B} \in \mathbb{R}^{sn \times sn}$ is a block diagonal matrix with $B(M_i)$, for $i = 1, \ldots, s$ in its diagonal, i.e.,
    \begin{align*}
    \mathbf{B} = 
    \begin{pmatrix}
         B(M_1)& \ldots & \mathbf{0}\\
        \vdots & \ddots & \vdots\\
        \mathbf{0} & \ldots & B(M_s)
    \end{pmatrix}.
    \end{align*}    

    \item $\mathbf{A}\in \mathbb{R}^{sn \times sn}$ is the block matrix
    \begin{align*}
        \mathbf{A} = 
        \begin{pmatrix}
            a_{11}\mathrm{Id} & \ldots & a_{1s}\mathrm{Id}\\
            \vdots & \ddots & \vdots\\
            a_{s1}\mathrm{Id} & \ldots & a_{ss}\mathrm{Id}
        \end{pmatrix}
    \end{align*}
    where $\mathrm{Id} \in \mathbb{R}^{n\times n}$ is the identity matrix.
\end{itemize}
Simple manipulation of Equation~\eqref{eq:rk:hamilton:matrix} implies 

\begin{align}
    \label{eq:block:implicit}
    \mathbf{W_0} = \left( \mathbf{Id} - h\mathbf{A} \mathbf{B}  \right) \mathbf{M} \left( \mathbf{Id} + h\mathbf{B}\mathbf{A}^T \right)
\end{align}
where $\mathbf{W_0} = \mathbf{g_0}^{\dagger}\mathbf{p_0}$ and $\mathbf{M} =\mathbf{G}^{\dagger}\mathbf{P}$. Notice that $\mathbf{M}$ carries the intermediate points $M_i$ on its diagonal. Therefore, finding $\mathbf{M}$ that solves the implicit Equation~\eqref{eq:block:implicit} is enough for the explicit update $W_1 = W_0 + h \sum_{i=0}^s b_i [B(M_i), M_i]$ from the ISOSyRK scheme to be computed. We call it the block algorithm for ISOSyRK schemes.

Clearly, Equation~\eqref{eq:ISOSyDIRK_implicit} is a particular case of Equation~\eqref{eq:block:implicit}, for the underlying Runge-Kutta method being the implicit midpoint rule. Therefore, the block algorithm can be seen as a generalization of the minimal-variable method introduced in~\cite{viviani2019minimal}. Notice that in the latter, the implicit equation to be solved consists of $n \times n$ matrices, while for the block algorithm, the matrices belong to $\mathbb{R}^{sn \times sn}$. Despite the increase in dimensionality, we show in Section~\ref{sec:numerical} that higher-order ISOSyRKs offer better conservation properties. Furthermore, the block algorithm inherits the simplicity of the minimal-variable algorithm proposed in~\cite{viviani2019minimal}, since a single implicit equation needs to be solved for each integration step. 

The simplicity of the block algorithm facilitates the implementation of general ISOSyRK schemes in comparision to the algorithm proposed in \cite{modin2019lie}. Equation~\eqref{eq:block:implicit} involves solely matrix-matrix multiplications which can be fully parallelized, in contrast to ISOSyDIRK schemes from~\cite{da2022variational} which have intrinsic serial aspects to it. 

The authors in~\cite{benzi2023solving} introduce efficient methods for solving Equation~\eqref{eq:ISOSyDIRK_implicit}. 
We describe similar methods to numerically solve Equation~\eqref{eq:block:implicit}. In the propositions below, we present proofs on why these methods work, which are simpler than those in~\cite{benzi2023solving}.
For an $s$-stage symplectic implicit Runge-Kutta scheme, let $\mathbf{W_0}, \mathbf{Id}, \mathbf{A}$ be the fixed matrices defined as above. For each matrix $\mathbf{X} \in \mathfrak{gl}(sn)$, denote by
$X_i$, its $i$-th diagonal $n \times n$ block. For $B: \mathfrak{gl}(n) \to \mathfrak{gl}(n)$ linear, let $\mathbf{B}: \mathfrak{gl}(sn) \to \mathfrak{gl}(sn)$ be the map where $\mathbf{B}(\mathbf{X})$ is a block matrix with $B(X_{i})$, for $i = 1, \ldots, s$ on its diagonal.

\begin{proposition}
There is $\epsilon >0$ and an open neighborhood $V_{\mathbf{W}_0} \subset \mathfrak{gl}(sn)$ of $\mathbf{W}_0$, such that, for $h \in (-\epsilon, \epsilon)$, 
\begin{align}
    \label{eq:proof_implicit}
    \mathbf{W_0} = \left( \mathbf{Id} - h\mathbf{A} \mathbf{B}(\mathbf{X})  \right) \mathbf{X} \left( \mathbf{Id} + h\mathbf{B}(\mathbf{X})\mathbf{A}^T \right)
\end{align}
has a unique solution  $\mathbf{X} \in V_{{\mathbf{W}_0}}$.    
\end{proposition}

\begin{proof}
Given some $\epsilon$ and $V_{\mathbf{W}_0}$ to be specified later, set 
\begin{align*}
 G: (\epsilon, \epsilon) \times V_{\mathbf{W}_0} &\to \mathfrak{gl}(sn) \\
(h, \mathbf{X}) &\mapsto \left( \mathbf{Id} - h\mathbf{A} \mathbf{B}(\mathbf{X})  \right) \mathbf{X} \left( \mathbf{Id} + h\mathbf{B}(\mathbf{X})\mathbf{A}^T \right) - \mathbf{W}_0 \ .     
\end{align*}
Clearly, $G(0, \mathbf{W}_0) = 0$. We use the implicit function theorem to show that $G^{-1}(0)$ is a curve through $(0, \mathbf{W}_0)$  transversal to the plane $h=0$. We must check that $\frac{\partial G}{\partial \mathbf{X}}(0, \mathbf{W}_0)$ is an invertible matrix. Indeed,
\[ DG(0, \mathbf{W}_0)(\delta h, \delta \mathbf{X}) \ = \  (  A B(\mathbf{W}_0) \mathbf{W}_0 -  \mathbf{W}_0 B(\mathbf{W}_0)A^T) \delta h + \delta \mathbf{X} \  \]
so that $\frac{\partial G}{\partial \mathbf{X}}(0, \mathbf{W}_0) = \mathbf{Id}$.
Thus the choice of $\epsilon$ and $V_{\mathbf{W}_0}$ are determined by the neighborhood of $(0, W_0)$ on which the implicit theorem holds. 
\end{proof}

\bigskip

We now prove that iteration obtains the solutions of \eqref{eq:proof_implicit}.

\begin{proposition} For a fixed, small $h$ and an open neighborhood of $\mathbf{W}_0$ contained in $V_{\mathbf{W}_0}$, the map
\[ C(\mathbf{X}) = \mathbf{X} - G(h, \mathbf{X})\]
is a contraction. In particular, a solution of Equation \eqref{eq:proof_implicit} is obtained by iterating $C$ starting with the initial condition $\mathbf{W}_0$.
\end{proposition}
\begin{proof}
For a fixed $h$, $G(h,\mathbf{X}) = 0$ means that $\mathbf{X}$ is a fixed point of $C$.
From the proof of the proposition above, $\frac{\partial G}{\partial \mathbf{X}}(0, \mathbf{W}_0) = \mathbf{Id}$ and since $G$ is a $C^1$ map, 
\[ \frac{\partial C}{\partial \mathbf{X}}(h, \mathbf{W}_0) = \mathbf{Id} - \frac{\partial G}{\partial \mathbf{X}}(h, \mathbf{W}_0)\]
is arbitrarily close to zero, for small $h$. Thus $C$ is a contraction keeping invariant a small ball around $\mathbf{W}_0$ and, as consequence, for $h$ small enough, the solution for Equation~\eqref{eq:proof_implicit} can be found with fixed point iterations
 \begin{align*}
 	\mathbf{X}_{n+1} = C(h, \mathbf{X}_n)
 \end{align*}
 starting with $\mathbf{X}_0 = \mathbf{W}_0$.
\end{proof}

A further investigation of how small $h$ needs to be for the map $C$ to be a contraction is left for future work. For the numerical solution of Equation~\eqref{eq:ISOSyDIRK_implicit}, matrix-matrix multiplication bounds the complexity of each fixed-point iteration step to $O(n^3)$. For higher-order ISOSyRKs, the block algorithm requires the dimensionality of the implicit Equation~\eqref{eq:block:implicit} to increase. However, for low-dimensional isospectral systems, $s$ and $n$ are comparable, and the numerical experiments in Section~\ref{sec:numerical} show that the time execution of the block algorithm does not suffer with the higher dimensionality when compared to ISOSyDIRKs of the same order. For high-dimensional isospectral systems, e.g., matrix fluid models~\cite{cifani2023efficient}, $s \ll n$ and the computational complexity of each fixed-point iteration step remains $O(n^3)$.

%% file: numerics.tex
\section{Numerical experiments}
\label{sec:numerical}

In this section, we compare the conservation properties of four ISOSyRK integrators:
\begin{itemize}
    \item \textbf{1s-SyDIRK}: Also known as the implicit midpoint rule, which is a $2^{nd}$-order ISOSyRK, implemented with the algorithm from~\cite{da2022variational,viviani2019minimal}.  
    \item \textbf{7s-SyDIRK}: A 7-stage $6^{th}$-order ISOSyDIRK~\cite{jiang2015sixth}, implemented with the algorithm from~\cite{da2022variational}.
    \item \textbf{2s-Gauss}: The 2-stage Gauss method~\cite{hairer2006geometric}, which is a $4^{th}$-order ISOSyRK, implemented with the block algorithm proposed above. 

    \item \textbf{3s-Gauss}: The 3-stage Gauss method~\cite{hairer2006geometric}, which is a $6^{th}$-order ISOSyRK, implemented with the block algorithm proposed above.
\end{itemize}
We run numerical tests for two isospectral Lie-Poisson systems, the rigid body and the periodic Toda lattice. For the former, we also test its generalized version in $n$-dimensions for $n=10,20$ and $50$. The block algorithm facilitates the implementation of the 3s-Gauss method, and its superior conservation properties illustrate the advantages of the block algorithm in comparison to the other algorithms.

\subsection{Rigid body}
\label{sec:numerical_rb}
The classical isospectral Lie-Poisson system is the ideal rigid body~\cite[Ch. 15]{Marsden1999}. In matrix form, $W \in \mathfrak{so}(3)$ is the angular momentum in the body frame and the Hamiltonian is defined as
\begin{align}
    \label{eq:rb_hamiltonian}
    \eta(W) = \frac{1}{2} \langle \mathcal{I}^{-1} W, W\rangle
\end{align}
where $\mathcal{I}$, the body inertia tensor, is a symmetric operator $\mathfrak{so}(3) \to \mathfrak{so}(3)$. The Lie-Poisson equations take the form
\begin{align}
    \label{eq:rb_lie-poisson}
    \dot{W} = -\left[ \mathcal{I}^{-1} W, W\right]
\end{align}
for $B(W) = \nabla \eta(W)^\dagger = -\mathcal{I}^{-1}W$.

\begin{figure}[h]
\centering
\includegraphics[trim={0.0 12.0 0.0 0.0},clip,width=\textwidth]{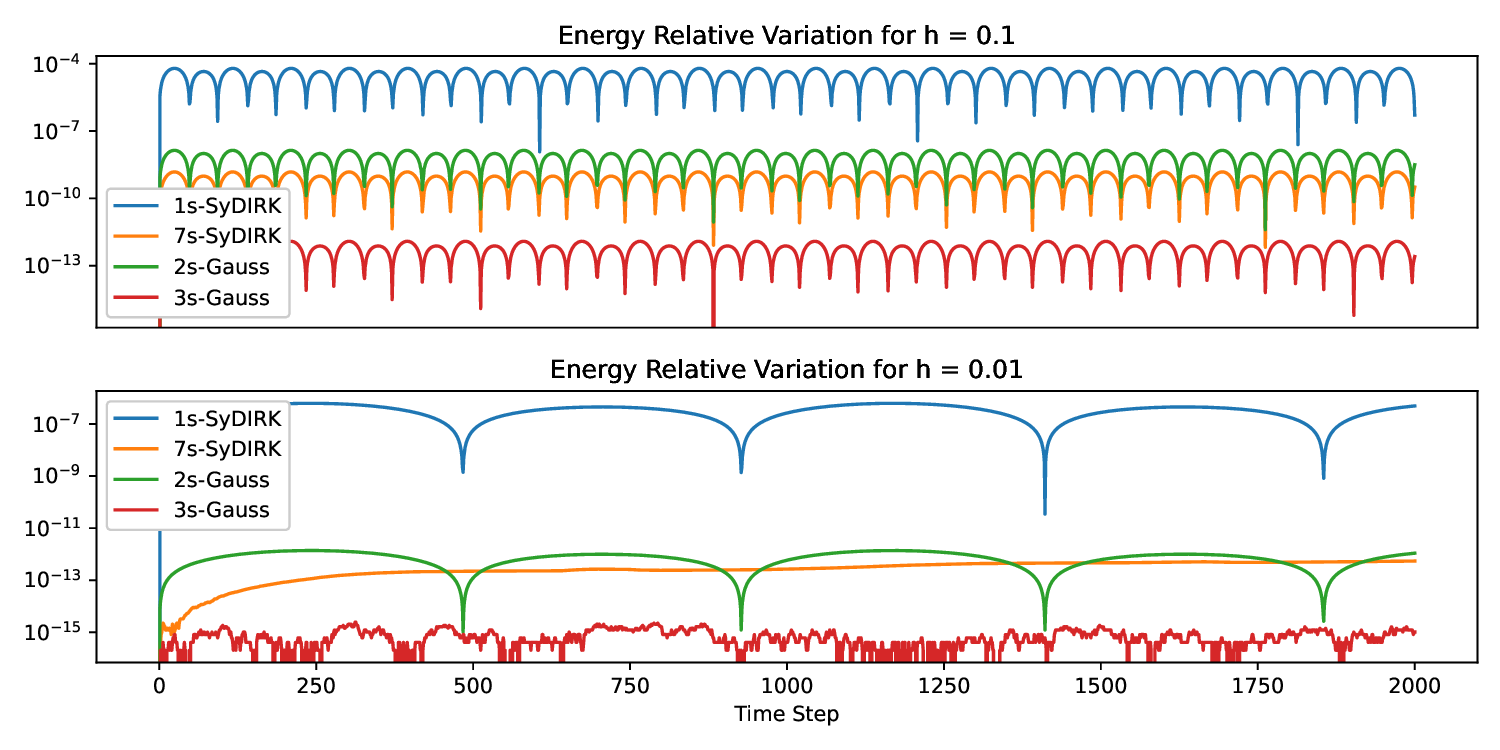} 
\caption{Energy conservation of the rigid body with four ISOSyRK integrators and time step $h =0.1$ (top) and $h = 0.01$ (bottom). We implemented 1s-SyDIRK and 7s-SyDIRK with the algorithm from~\cite{da2022variational,viviani2019minimal} and 2s-Gauss and 3s-Gauss with the block algorithm. Table~\ref{tab:rb_times} shows the time execution for each simulation.}
\label{fig:rb_hamiltonian_preservation}
\end{figure}

Figure~\ref{fig:rb_hamiltonian_preservation} shows the relative Hamiltonian variation for $N = 2000$ time steps for step sizes of $h = 0.1$ and $h=0.01$. The initial condition $W_0$ is a randomly generated skew-symmetric matrix 
and we used fixed-point iterations with tolerance $10^{-15}$ (for the Fröbenius norm) for all integrators. As expected, increasing the order of the time-integrator results in better energy conservation. However, Table~\ref{tab:rb_times} indicates that the block algorithm for the $3s$-Gauss method is less time-consuming than the same order $7s$-SyDIRK implemented with the algorithm proposed in~\cite{da2022variational}. Despite the higher dimension of the implicit equations solved by fixed-point iterations, Table~\ref{tab:rb_iterations} shows that the number of iterations needed to achieve a tolerance of $10^{-15}$  does not increase significantly from the $\textbf{1s-SyDIRK}$ integrator to the  $\textbf{2s-Gauss}$ and $\textbf{3s-Gauss}$ integrators. It also shows that there seems to exist a trade-off between smaller time steps and faster convergence of the fixed-point solver.

\begin{table}[h]
    \centering
\begin{tabular*}{\textwidth}{@{\extracolsep\fill}lcccc}
        \hline
         & \textbf{1s-SyDIRK} & \textbf{7s-SyDIRK} & \textbf{2s-Gauss} & \textbf{3-Gauss} \\ \hline
        $h = 0.1$ & 0.8473 & 6.0511 & 3.1222 & 4.0645 \\ \hline
        $h = 0.01$ & 0.6170 & 4.7073 & 2.4780 & 4.1032 \\ \hline
\end{tabular*}
    \caption{Time execution in seconds for each simulation of the rigid body system.}
    \label{tab:rb_times}
\end{table}
\begin{table}[h]
    \centering
\begin{tabular*}{\textwidth}{@{\extracolsep\fill}lccc}
        \hline
         & \textbf{1s-SyDIRK} & \textbf{2s-Gauss} & \textbf{3s-Gauss} \\ \hline
        $h = 0.1$ & 8 & 11 & 10 \\ \hline
        $h = 0.01$ & 5 &  6 & 6 \\ \hline
    \end{tabular*}
    \caption{Maximum number of fixed point iterations to achieve a tolerance of $10^{-15}$ as a solution of the implicit equations with the $\textbf{1s-SyDIRK}$, $\textbf{2s-Gauss}$ and $\textbf{3s-Gauss}$ integrators for the simulation of the rigid body system. The $\textbf{7s-SyDIRK}$ integrator is left out since it is a composition of \textbf{1s-SyIRK} steps with scaled $h$.}
    \label{tab:rb_iterations}
\end{table}

\subsection{Toda lattice}
\label{sec:numerical_toda}

An isospectral system that is not initially Hamiltonian but can be extended to a Lie-Poisson system in $\mathfrak{gl}(n)$ is the periodic Toda lattice in its Lax pair formulation~\cite{Toda1967,deift1983ordinary,modin2019lie}. The Hamiltonian $\eta:\mathfrak{gl}(n)\to\mathbb{R}$ is given by
\begin{align}
    \eta(W) = 2 \mathrm{Tr}(W^2) - \frac{1}{2}\mathrm{Tr}(W^{\dagger}B(W)) 
\end{align}
where
\begin{align*}
    B(W) = 
    \begin{pmatrix}
        0 & W_{12} & 0 & \cdots & - W_{1n}\\
        -W_{21} & 0 & W_{23} & \cdots & 0 \\
        0 & -W_{32} & 0 & \cdots & 0 \\
        \vdots & \vdots & \vdots & \ddots & \vdots\\
        W_{n1} & 0 & 0 & \cdots & 0
    \end{pmatrix}
\end{align*}
The dynamic equation is of the form of Equation~\eqref{eq:isospectral_lie_poisson} and in Fig.~\ref{fig:toda_hamiltonian_preservation} we show the conservation results for the extended periodic Toda lattice with $N = 1000$ time steps and step sizes of $h = 0.1$ and $h=0.01$. Here we use $n=4$ with the initial condition given by
\begin{align*}
    W_0 = \begin{pmatrix}
         -1 & -1 & 0 & 1\\
         -1 & 1 & 1 & 0\\
         0 & 1  & -1 & -1\\
         1 & 0 & -1 & 1
    \end{pmatrix},
\end{align*} 
which is also used in the numerical experiments for the Toda lattice in~\cite{modin2019lie,viviani2019minimal,da2022variational}. We used fixed-point iterations to solve the required implicit equations with tolerance $10^{-14}$. The maximum number of iterations to achieve the tolerance is described in Table~\ref{tab:toda_iterations} for each integrator. 
As in the previous experiment, higher order time-integrators result in better energy conservation. Furthermore, Table~\ref{tab:toda_times} indicates that the block algorithm implementation of $\textbf{3s-Gauss}$ is less time-consuming than the same order $\textbf{7s-SyDIRK}$ implemented with the algorithm proposed in~\cite{da2022variational}, despite the higher dimension of its implicit equation.

\begin{figure}[h]
\centering
\includegraphics[width=\textwidth]{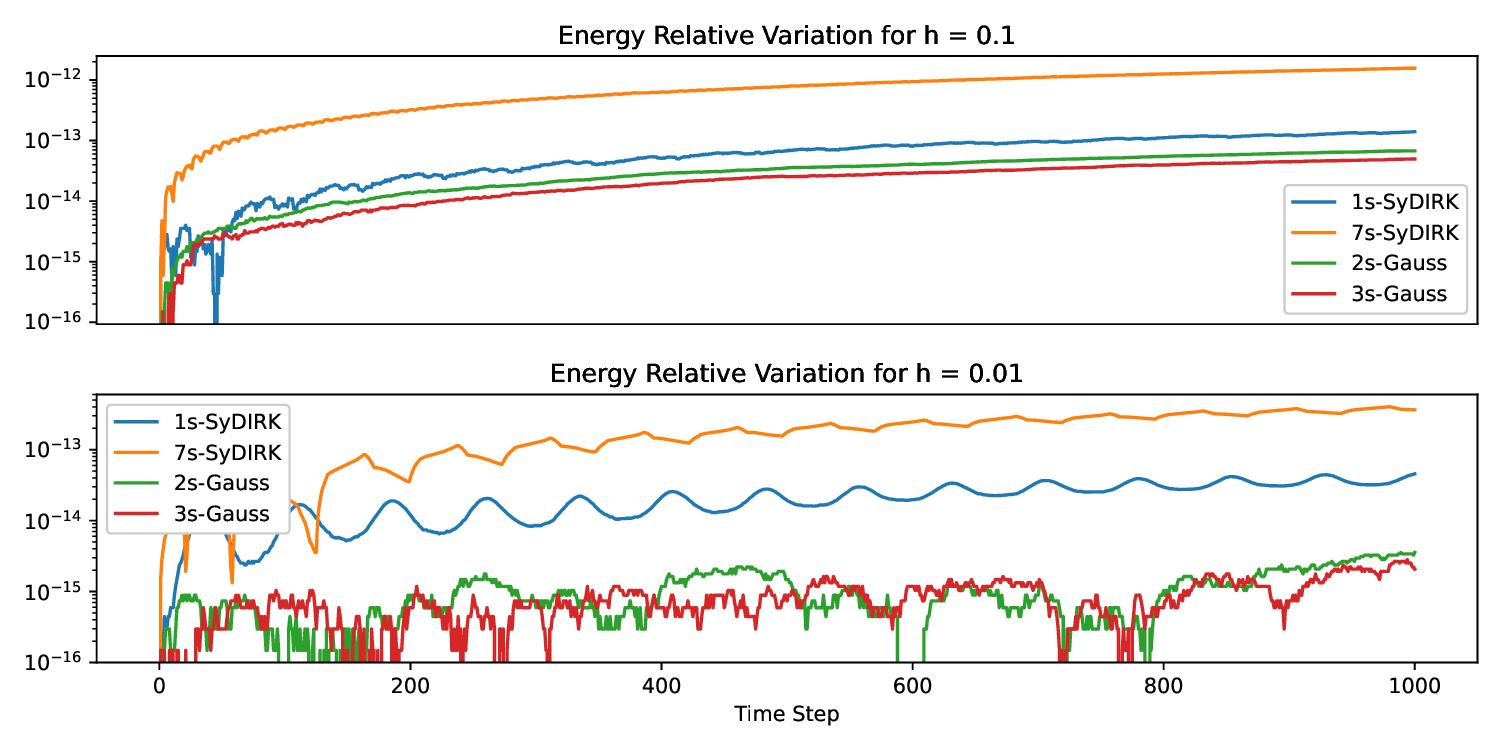} 
\caption{Top: Energy conservation of the Toda lattice with four ISOSyRK integrators and time step $h =0.1$. Bottom: Energy conservation for four ISOSyRK integrators with time step h = 0.01. We implemented 1s-SyDIRK and 7s-SyDIRK with the algorithm from~\cite{da2022variational,viviani2019minimal} and 2s-Gauss and 3s-Gauss with the block algorithm. Table~\ref{tab:toda_times} shows the time execution for each simulation.}
\label{fig:toda_hamiltonian_preservation}
\end{figure}

\begin{table}[h!]
    \centering
\begin{tabular*}{\textwidth}{@{\extracolsep\fill}lcccc}        
	\hline
         & \textbf{1s-SyDIRK} & \textbf{7s-SyDIRK} & \textbf{2s-Gauss} & \textbf{3-Gauss} \\ \hline
        $h = 0.1$ & 0.9575 & 6.0116 & 2.0907 & 2.8484 \\ \hline
        $h = 0.01$ & 0.4197 & 2.7707 & 1.1158 & 1.5020 \\ \hline
\end{tabular*}
    \caption{Time execution for each in seconds simulation of the Toda Lattice.}
    \label{tab:toda_times}
\end{table}

\begin{table}[h!]
    \centering
\begin{tabular*}{\textwidth}{@{\extracolsep\fill}lccc}       
        \hline
         & \textbf{1s-SyDIRK} & \textbf{2s-Gauss} & \textbf{3s-Gauss} \\ \hline
        $h = 0.1$ & 23 & 17 & 16 \\ \hline
        $h = 0.01$ & 8 &  8 & 8 \\ \hline
\end{tabular*}
    \caption{Maximum number of fixed point iterations to achieve a tolerance of $10^{-14}$ as a solution of the implicit equations with the $\textbf{1s-SyDIRK}$, $\textbf{2s-Gauss}$ and $\textbf{3s-Gauss}$ integrators for the simulation of the Toda lattice. The $\textbf{7s-SyDIRK}$ integrator is left out since it is a composition of \textbf{1s-SyIRK} steps with scaled $h$.}
    \label{tab:toda_iterations}
\end{table}

\subsection{Generalized Rigid Body}

In order to illustrate the effects of increasing dimensionality on the performance, we compare the integrators for the generalized rigid body. The dynamics is analogous to the classical rigid body system from Section~\ref{sec:numerical_rb} with the difference that the Lie algebra over which the system is defined is $\mathfrak{so}(n)$ instead of $\mathfrak{so}(3)$. We use $n = 10, 20, 50$ with $N = 2000$ time steps, step size $h = 0.01$ and tolerance of $10^{-14}$ for the fixed-point iterations (the Fröbenius norm). 
Figure~\ref{fig:generalized_rb_hamiltonian_preservation} shows that the higher-order integrators display better conservation of the Hamiltonian for all dimensions. That indicates that using higher-order integrators, such as $\textbf{3s-Gauss}$, are better suited for long-time simulations.
Tables~\ref{tab:generalized_rb_times} and~\ref{tab:generalized_rb_iterations} show that increasing the dimension of the system, combined with the high-dimensional nature of the block-algorithm, makes the average time-integration for each time step slower than the ISOSyDIRKs. However, as already mentioned, parallelization could mitigate that issue.

\begin{figure}[h]
\centering
\includegraphics[width=\textwidth]{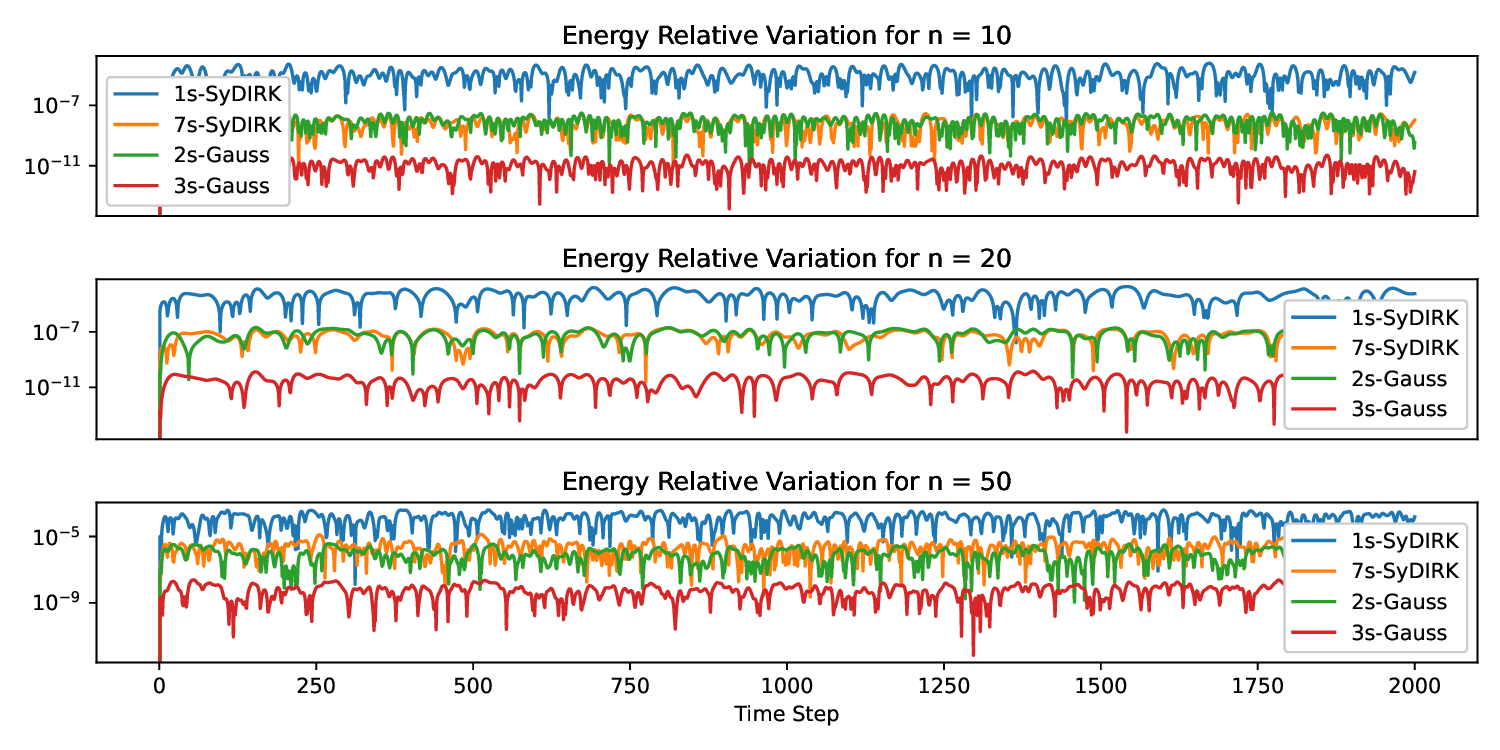} 
\caption{Energy conservation of the generalized rigid body with four ISOSyRK integrators for $n = 10$ (top), $n=20$ (middle) and $n = 50$ (bottom). We implemented 1s-SyDIRK and 7s-SyDIRK with the algorithm from~\cite{da2022variational,viviani2019minimal} and 2s-Gauss and 3s-Gauss with the block algorithm. Table~\ref{tab:generalized_rb_times} shows the time execution for each simulation.}
\label{fig:generalized_rb_hamiltonian_preservation}
\end{figure}

\begin{table}[h!]
    \centering
\begin{tabular*}{\textwidth}{@{\extracolsep\fill}lcccc}        
\hline
         & \textbf{1s-SyDIRK} & \textbf{7s-SyDIRK} & \textbf{2s-Gauss} & \textbf{3-Gauss} \\ \hline
        $n = 10$ & 1.2789 & 8.8158 & 3.7640 & 5.6305 \\ \hline
        $n = 20$ & 1.4853 & 10.1409 & 7.2950 & 14.3572 \\ \hline
        $n = 50$ & 9.0200 & 68.1599 & 57.1546 & 112.5086 \\ \hline

    \end{tabular*}
    \caption{Time execution in seconds for each simulation of the generalized rigid body for $n=10, 20$ and $50$.}
    \label{tab:generalized_rb_times}
\end{table}

\begin{table}[h!]
    \centering
\begin{tabular*}{\textwidth}{@{\extracolsep\fill}lccc}
        \hline
         & \textbf{1s-SyDIRK} & \textbf{2s-Gauss} & \textbf{3s-Gauss} \\ \hline
        $n = 10$ & 15 & 11 & 11 \\ \hline
        $n = 20$ & 11 &  14 & 13 \\ \hline
        $n = 50$ & 21 &  24 & 21 \\ \hline

    \end{tabular*}
    \caption{Maximum number of fixed point iterations to achieve a tolerance of $10^{-14}$ as a solution of the implicit equations with the $\textbf{1s-SyDIRK}$, $\textbf{2s-Gauss}$ and $\textbf{3s-Gauss}$ integrators for the simulation of the generalized rigid body with $n=10, 20$ and $50$. The $\textbf{7s-SyDIRK}$ integrator is left out since it is a composition of \textbf{1s-SyIRK} steps with scaled $h$.}
    \label{tab:generalized_rb_iterations}
\end{table}

\newpage

%% file: conclusion.tex
\section{Conclusions and Future Work}
\label{sec:conclusion}
We introduced the block algorithm for ISOSyRK methods. It uses block matrices to describe an implementation of ISOSyRK schemes that is simpler than the previous approach from~\cite{modin2019lie}. As the authors in~\cite{benzi2023solving} did for the particular case of the implicit midpoint rule, we also show that the implicit equation resulting from the block algorithm can be solved efficiently using fixed-point iterations.

For low-dimensional isospectral Lie-Poisson systems, when compared to ISOSyDIRK methods from~\cite{da2022variational} of the same order, our approach using the $6^{th}$-order 3-stage Gauss method results in better energy conservation and time execution than the $6^{th}$-order 7-stage SyDIRK from~\cite{jiang2015sixth}. For high-dimensional systems, the block algorithm still offers better Hamiltonian conservation. However, it seems to suffer from an increase in the time execution for each time step compared to the same-order ISOSyDIRK schemes. To optimize this trade-off, parallelization can be used. We plan to investigate this in future work.

We also intend to investigate whether the block algorithm can help isospectral systems for fluid dynamics problems to achieve better conservation properties. The authors in~\cite{cifani2023efficient} developed an efficient framework for the time integration of the ideal fluid dynamics on the sphere.  Their spatial discretization is an isospectral system resulting of the quantization of the Euler equations~\cite{Zeitlin1991,Hoppe1998}. They use the ISOSyRK associated with the implicit midpoint rule and, with the block algorithm, we believe we can implement their framework using higher-order ISOSyRK methods to possibly achieve better conservation of the Hamiltonian and the Casimirs of the system such as enstrophy. Their framework relies on the parallelization of several aspects of their algorithm. Since the block algorithm is fully parallelizable, we can also exploit the computational infrastructure to enjoy a more advantageous trade-off between the high-dimensionality and better conservation properties resulting from higher-order algorithms.

Recently, an isospectral discretization was proposed to another spherical fluid model: the shallow water equation~\cite{franken2024zeitlin}. Again, the implicit midpoint rule is used as a time integrator and we are interested in how higher-order ISOSyRK schemes, easily implemented with the block algorithm, can improve conservation results while still offering efficiency.